\documentclass[11pt]{article}

\usepackage[margin=1in]{geometry}
\usepackage{amsmath,amssymb,amsthm,mathtools}
\usepackage{enumitem}
\usepackage{microtype}
\usepackage[colorlinks=true,linkcolor=blue,citecolor=blue,urlcolor=blue]{hyperref}
\hypersetup{
  pdftitle={The Expiring Coupon Collector: Sliding-Window Surjection Flux and Rare-Entry Laws},
  pdfauthor={Christopher D. Long, Headlamp Software}
}
\newtheorem{theorem}{Theorem}[section]
\newtheorem{lemma}[theorem]{Lemma}
\newtheorem{proposition}[theorem]{Proposition}
\newtheorem{corollary}[theorem]{Corollary}
\theoremstyle{definition}

\newtheorem{remark}[theorem]{Remark}

\newcommand{\Pbb}{\mathbb P}
\newcommand{\Ebb}{\mathbb E}

\newcommand{\Rbb}{\mathbb R}
\newcommand{\1}{\mathbf 1}

\newcommand{\Exp}{\mathrm{Exp}}

\newcommand{\Bin}{\mathrm{Bin}}
\newcommand{\Var}{\mathrm{Var}}

\newcommand{\floor}[1]{\left\lfloor #1\right\rfloor}

\title{The Expiring Coupon Collector\\
Sliding-Window Surjection Flux and Rare-Entry Laws}
\author{Christopher D. Long\\
Headlamp Software\\
\texttt{galizur@gmail.com}}
\date{}

\begin{document}
\maketitle

\begin{abstract}
We study the coupon collector with deterministic expiration: one coupon is drawn at each time, and each coupon remains active for exactly $M$ draws.  Completion occurs when all $n$ coupon types are simultaneously active.  Equivalently, the current length-$M$ sliding window of draws must contain all $n$ types.

The central object is not the one-time probability that a random window is onto, but the stationary flux of new entries into the onto-window set.  We compute this flux exactly:
\[
        \mu_{n,M}
        =\Pbb(W_{t-1}\text{ is not onto},\ W_t\text{ is onto})
        =\frac{(n-1)(n-1)!S(M-1,n-1)}{n^M},
\]
where $S(\cdot,\cdot)$ denotes a Stirling number of the second kind.  Under a quantitative subcritical separation condition, satisfied in particular by every fixed integer scale $M=\floor{\alpha n\log n}$, $0<\alpha<1$, we prove local declumping and obtain
\[
        \mu_{n,M}T_{n,M}\Rightarrow \Exp(1).
\]
For the fixed subcritical scale $M=\floor{\alpha n\log n}$, $0<\alpha<1$, this gives the logarithmic scale
\[
        \log T_{n,M}=n^{1-\alpha}+o_{\mathbb P}(n^{1-\alpha}),
        \qquad
        \log \Ebb T_{n,M}=n^{1-\alpha}+o(n^{1-\alpha}),
\]
and, when $\alpha>1/2$, the sharper normalization
\[
        n^{-\alpha}e^{-n^{1-\alpha}}T_{n,M}\Rightarrow \Exp(1),
        \qquad
        \Ebb T_{n,M}\sim n^\alpha e^{n^{1-\alpha}}.
\]
Thus the leading scale proposed in the Math StackExchange discussion is made rigorous; the exact finite-$n$ flux gives the canonical normalization throughout the subcritical range.  The result is a sliding-window companion to rare-void entry-flux methods for nonmonotone coupon collectors.
\end{abstract}

\section{Introduction}

The classical coupon collector is monotone: once a type has appeared, that part of the task remains complete.  The expiring coupon collector is nonmonotone.  One coupon is drawn at each time, uniformly from $[n]=\{1,\ldots,n\}$, and the active coupons at time $t$ are precisely those appearing in the last $M$ draws.  The process is complete when the currently active coupons include all $n$ types.

This model was posed on Math StackExchange by mjqxxxx \cite{MSEQuestion}.  The question asks for the expected time to completion and emphasizes the sharp contrast between the classical $\Theta(n\log n)$ scale and the case $M=n$, where a perfect run of $n$ distinct coupons is required.  Joriki's answer \cite{MSEAnswer} gives a detailed heuristic for $M=\floor{\alpha n\log n}$, $0<\alpha<1$, predicting the scale
\[
        n^\alpha\exp(n^{1-\alpha}).
\]
The answer explicitly notes that the approximations are not fully justified, although numerical evidence supports the prediction.

The purpose of this note is to give a clean entry-flux formulation of the problem and to record the resulting rare-entry theorem.  At time $t$, the active collection is exactly the set of types appearing in the last $M$ draws.  Thus completion is the event that a length-$M$ word over $[n]$ is onto.  The target set is not absorbing, and visits occur in overlapping-window clusters.  The correct rare-event normalization is therefore not simply the stationary mass of the onto-window set, but the stationary rate of \emph{new entries} into that set.

This is the same principle underlying related rare-void work on clumsy and careless coupon collectors \cite{LongClumsyCareless}.  In that setting, the schematic principle is
\[
\boxed{
\text{stationary entry flux}+\text{fast mixing}+\text{local declumping}
\Longrightarrow \text{exponential hitting law}.}
\]
The present model supplies a third basic mechanism.  For clumsy collectors the flux comes from a product stationary law; for careless collectors it comes from a lucky-climb high tail; for expiring collectors it comes from a sliding-window surjection count.

\subsection*{Main contribution}

Let
\[
        W_t=(X_{t-M+1},\ldots,X_t)\in[n]^M
\]
be the current full window in the stationary two-sided process, where the $X_i$ are iid uniform on $[n]$.  Let $A_{n,M}$ be the set of onto words in $[n]^M$.  For the one-sided collector started empty, define
\[
        T_{n,M}:=\inf\{t\ge n:\{X_{\max(1,t-M+1)},\ldots,X_t\}=[n]\}.
\]
Thus, once \(t\ge M\), completion is exactly the event \(W_t\in A_{n,M}\).  The stationary law of \(W_t\) is uniform on \([n]^M\), and
\[
        \Pbb(W_t\in A_{n,M})=\frac{n!S(M,n)}{n^M}.
\]
Here and throughout, \(S(\cdot,\cdot)\) denotes a Stirling number of the second kind.
The new-entry flux is
\[
        \mu_{n,M}:=\Pbb(W_{t-1}\notin A_{n,M},\ W_t\in A_{n,M}).
\]
Our first result is the exact identity
\[
\boxed{
        \mu_{n,M}=\frac{(n-1)(n-1)!S(M-1,n-1)}{n^M}.}
\]
Equivalently,
\[
        \mu_{n,M}
        =\left(1-\frac1n\right)^M
          \frac{(n-1)!S(M-1,n-1)}{(n-1)^{M-1}}.
\]
The second form makes the mechanism transparent: the new symbol must be the unique missing type of the previous window, and the middle $M-1$ symbols must already contain all other $n-1$ types.

The rare-entry theorem then gives
\[
        \mu_{n,M}T_{n,M}\Rightarrow\Exp(1)
\]
in subcritical rare-window regimes.  In particular, if
\[
        M=\floor{\alpha n\log n},
        \qquad 0<\alpha<1,
\]
then the exact flux normalization gives
\[
        \log \mu_{n,M}=-n^{1-\alpha}+o(n^{1-\alpha}),
        \qquad
        \log \Ebb T_{n,M}=n^{1-\alpha}+o(n^{1-\alpha}).
\]
For \(\alpha>1/2\), the familiar simpler ratio form is valid:
\[
        n^{-\alpha}e^{-n^{1-\alpha}}T_{n,M}\Rightarrow\Exp(1),
        \qquad
        \Ebb T_{n,M}\sim n^\alpha e^{n^{1-\alpha}},
\]
with the corresponding variance asymptotic.

\section{Sliding-window chain}

Let $(X_t)_{t\in\mathbb Z}$ be iid uniform random variables on $[n]$.  For $M\ge n$, define the length-$M$ window
\[
        W_t=(X_{t-M+1},\ldots,X_t).
\]
The two-sided stationary construction is convenient for flux computations.  The process $(W_t)_{t\in\mathbb Z}$ is the Markov chain on $[n]^M$ that shifts left by one coordinate and appends a fresh uniform symbol.  It has the uniform stationary law on $[n]^M$.

Let
\[
        A_{n,M}:=\{w\in[n]^M:\text{ every symbol in }[n]\text{ appears in }w\}
\]
be the onto-window target set.  For the stationary chain, write
\[
        T^{\rm stat}_{n,M}:=\inf\{t\ge1:W_t\in A_{n,M}\}.
\]
For a one-sided process started empty at time \(0\), the completion time is
\[
        T_{n,M}:=\inf\{t\ge n:\{X_{\max(1,t-M+1)},\ldots,X_t\}=[n]\}.
\]
Completion by time \(M\) means that the initial prefix has already collected all types.  In the rare-window regimes treated below this early event has negligible probability, so the one-sided process can be compared with the stationary window chain after the deterministic burn-in time \(M\).

\begin{proposition}[Stationary target mass]\label{prop:mass}
For $M\ge n$,
\[
        \pi_{n,M}:=\Pbb(W_t\in A_{n,M})
        =\frac{n!S(M,n)}{n^M}
        =\sum_{j=0}^n(-1)^j\binom{n}{j}\left(1-\frac jn\right)^M.
\]
\end{proposition}

\begin{proof}
There are $n!S(M,n)$ onto words of length $M$ over an alphabet of size $n$.  Division by the total number $n^M$ gives the Stirling-number expression.  The inclusion-exclusion formula is the usual count of words missing none of the $n$ symbols.
\end{proof}

\section{Exact stationary entry flux}

Define the new-entry indicator
\[
        E_t:=\1\{W_{t-1}\notin A_{n,M},\ W_t\in A_{n,M}\}
\]
and the stationary entry flux
\[
        \mu_{n,M}:=\Pbb(E_t=1).
\]

\begin{theorem}[Exact sliding-window surjection flux]\label{thm:flux}
For every $M\ge n\ge2$,
\[
        \mu_{n,M}
        =\frac{(n-1)(n-1)!S(M-1,n-1)}{n^M}.
\]
Equivalently,
\[
        \mu_{n,M}
        =\left(1-\frac1n\right)^M
          \frac{(n-1)!S(M-1,n-1)}{(n-1)^{M-1}}.
\]
\end{theorem}

\begin{proof}
Write
\[
        W_{t-1}=(X_{t-M},X_{t-M+1},\ldots,X_{t-1}),
        \qquad
        W_t=(X_{t-M+1},\ldots,X_{t-1},X_t).
\]
Let
\[
        B_t=(X_{t-M+1},\ldots,X_{t-1})
\]
be the common middle block of length $M-1$.

Suppose first that $E_t=1$.  Since $W_t$ is onto and $W_{t-1}$ is not onto, the appended symbol $X_t$ must be a type that is absent from the middle block $B_t$.  Moreover, $B_t$ must contain all the other $n-1$ types; otherwise appending one symbol could not make $W_t$ onto.  Finally, the departing symbol $X_{t-M}$ cannot be the missing type of $B_t$, because if it were, then $W_{t-1}$ would also be onto.

Conversely, if for some $a\in[n]$ the middle block $B_t$ contains every type in $[n]\setminus\{a\}$ and contains no $a$, if $X_t=a$, and if $X_{t-M}\ne a$, then $W_t$ is onto while $W_{t-1}$ is not.  Therefore these conditions are necessary and sufficient.

For a fixed $a$, the probability that $B_t$ is an onto word over the alphabet $[n]\setminus\{a\}$ is
\[
        \frac{(n-1)!S(M-1,n-1)}{n^{M-1}}.
\]
Independently, $\Pbb(X_t=a)=1/n$ and $\Pbb(X_{t-M}\ne a)=(n-1)/n$.  Summing over the $n$ possible choices of $a$ gives
\[
        \mu_{n,M}
        =n\cdot \frac{(n-1)!S(M-1,n-1)}{n^{M-1}}
             \cdot \frac1n\cdot \frac{n-1}{n}
        =\frac{(n-1)(n-1)!S(M-1,n-1)}{n^M}.
\]
The alternative expression follows by factoring
\[
        \frac{(n-1)(n-1)!S(M-1,n-1)}{n^M}
        =\left(1-\frac1n\right)^M
          \frac{(n-1)!S(M-1,n-1)}{(n-1)^{M-1}}.
\]
\end{proof}

\begin{remark}[Flux versus mass]
The stationary mass $\pi_{n,M}$ is the probability that a random window is onto.  The entry flux $\mu_{n,M}$ is the probability that the window becomes onto at the current shift after not being onto at the previous shift.  These are different because onto windows occur in overlapping clumps.  The normalization for the first hitting time is governed by the flux.
\end{remark}

\section{Finite-dependence rare-entry theorem}

This section gives a self-contained rare-entry principle adapted to the sliding-window process.  It is a finite-dependence specialization of the rare-void hitting principle used in related nonmonotone coupon-collector models \cite{LongClumsyCareless}.

For the stationary two-sided process define
\[
        \theta_{n,M}:=\sum_{u=1}^{M}
        \Pbb(E_u=1\mid E_0=1).
\]
This is the expected number of further entries in the dependence range of a given entry.  The condition $\theta_{n,M}\to0$ says that entries are asymptotically declumped.

For a deterministic window \(w\in[n]^M\), let \(\Pbb_w\) denote the law of the sliding-window chain started from \(W_0=w\), and write
\[
        H_A:=\inf\{t\ge0:W_t\in A_{n,M}\}
\]
for the corresponding target-hitting time.

\begin{theorem}[Finite-dependence rare-entry law]\label{thm:finite-dependence}
Assume that, along a sequence $n\to\infty$ with $M=M_n\ge n$,
\[
        \pi_{n,M}\to0,
        \qquad
        \mu_{n,M}\to0,
        \qquad
        M\mu_{n,M}\to0,
        \qquad
        \theta_{n,M}\to0.
\]
Let $T_{n,M}$ be the first completion time for the one-sided expiring collector started empty.  Then
\[
        \mu_{n,M}T_{n,M}\Rightarrow\Exp(1).
\]
If, in addition, there are integers \(b=b_n\) and a constant \(c>0\) such that
\[
        M=o(b),
        \qquad
        b\mu_{n,M}\to0,
        \qquad
        \inf_w \Pbb_w(H_A\le b+M+1)\ge c b\mu_{n,M}
\]
for all sufficiently large \(n\), then for every fixed \(r\ge1\),
\[
        \Ebb(\mu_{n,M}T_{n,M})^r\to r!,
\]
and hence
\[
        \Ebb T_{n,M}\sim\frac1{\mu_{n,M}},
        \qquad
        \Var(T_{n,M})\sim\frac1{\mu_{n,M}^2}.
\]
\end{theorem}

\begin{proof}
We first prove the exponential law under the stationary window process.  Let
\[
        T^{\rm ent}_{n,M}:=\inf\{t\ge1:E_t=1\}
\]
be the first stationary entry time.  Choose integers $b=b_n$ such that
\[
        M=o(b),
        \qquad
        b\mu_{n,M}\to0,
        \qquad
        \frac1{b\mu_{n,M}}\to\infty.
\]
For example one may take $b=\lfloor M^{1/2}\mu_{n,M}^{-1/2}\rfloor$ after a harmless adjustment, using $M\mu_{n,M}\to0$.

Let
\[
        N_b:=\sum_{t=1}^{b}E_t.
\]
Then
\[
        \Ebb N_b=b\mu_{n,M}.
\]
The event $E_t$ depends only on the symbols $X_{t-M},\ldots,X_t$.  Hence $E_s$ and $E_t$ are independent whenever $|s-t|>M$.  Therefore
\[
\begin{aligned}
        \Ebb[(N_b)_2]
        &=2\sum_{1\le s<t\le b}\Pbb(E_s=1,E_t=1)\notag\\
        &\le 2b\mu_{n,M}\sum_{u=1}^{M}\Pbb(E_u=1\mid E_0=1)
            +2b^2\mu_{n,M}^2\notag\\
        &=2b\mu_{n,M}\theta_{n,M}+2b^2\mu_{n,M}^2
        =o(b\mu_{n,M}).
\end{aligned}
\]
Since
\[
        0\le \Ebb N_b-\Pbb(N_b\ge1)\le \Ebb[(N_b)_2],
\]
we have the one-block law
\[
        \Pbb(N_b\ge1)=b\mu_{n,M}(1+o(1)).
\]

Now separate active blocks of length $b$ by gaps of length $M+1$.  Events in distinct active blocks are independent after the separating gaps, because they depend on disjoint sets of underlying iid symbols.  For fixed $x>0$, put
\[
        K_n=\left\lfloor\frac{x}{b\mu_{n,M}}\right\rfloor.
\]
The probability that no entry occurs in the $K_n$ active blocks is
\[
        \left(1-b\mu_{n,M}(1+o(1))\right)^{K_n}\to e^{-x}.
\]
The total expected number of entries in the separating gaps is at most
\[
        K_n(M+1)\mu_{n,M}
        \le \frac{x(M+1)}{b}(1+o(1))\to0,
\]
so entries in the gaps are negligible.  The total time used by the active blocks and gaps is
\[
        K_n(b+M+1)=\frac{x}{\mu_{n,M}}(1+o(1)).
\]
Thus
\[
        \Pbb_{\mathrm{stat}}(\mu_{n,M}T^{\rm ent}_{n,M}>x)\to e^{-x}.
\]
Because \(\pi_{n,M}\to0\), the stationary process starts outside the target with probability tending to one; on that event, the first hit of \(A_{n,M}\) is exactly a new entry.  Hence
\[
        \Pbb_{\mathrm{stat}}(\mu_{n,M}T^{\rm stat}_{n,M}>x)\to e^{-x}.
\]

It remains to pass from the stationary chain to the process started empty.  After $M$ draws, the current length-$M$ window \(W_M=(X_1,\ldots,X_M)\) is iid uniform on \([n]^M\), hence exactly stationary.  Moreover, completion by time \(M\) is exactly the event \(W_M\in A_{n,M}\), so
\[
        \Pbb(T_{n,M}\le M)=\pi_{n,M}\to0.
\]
Conditioned on no completion by time \(M\), the law of \(W_M\) is the uniform stationary law conditioned on \(A_{n,M}^c\).  Its total-variation distance from the unconditioned stationary law is \(\pi_{n,M}\).  Therefore the post-burn-in hitting distribution differs from the stationary one by \(o(1)\), uniformly in the later time horizon.  Also \(M\mu_{n,M}\to0\), so the deterministic burn-in time is negligible on the \(\mu_{n,M}^{-1}\) scale.  This gives
\[
        \mu_{n,M}T_{n,M}\Rightarrow\Exp(1).
\]

The moment assertion follows from the same block argument once the displayed
uniform one-block lower bound is available.  Put
\[
        L_n:=b+M+1,
        \qquad
        p_n:=c b\mu_{n,M}.
\]
The strong Markov property at the deterministic times \(0,L_n,2L_n,\ldots\)
gives, for all large \(n\),
\[
        \sup_w \Pbb_w(H_A>kL_n)
        \le (1-p_n)^k,
        \qquad k\ge0 .
\]
Since \(M=o(b)\) and \(b\mu_{n,M}\to0\), we have
\(\mu_{n,M}L_n=(1+o(1))b\mu_{n,M}\).  Hence, for each \(x\ge0\), choosing
\(k=\lfloor x/(\mu_{n,M}L_n)\rfloor\) yields constants \(C,c_0>0\), independent of
\(n\), such that
\[
        \sup_w\Pbb_w(\mu_{n,M}H_A>x)
        \le C e^{-c_0x} .
\]
The deterministic burn-in contributes only \(M\mu_{n,M}=o(1)\), so the same
exponential domination applies to \(\mu_{n,M}T_{n,M}\) after enlarging \(C\).
Therefore the family \(\{(\mu_{n,M}T_{n,M})^r:n\ge1\}\) is uniformly
integrable for every fixed \(r\ge1\), and the already-proved convergence in
law to \(\Exp(1)\) gives
\(\Ebb(\mu_{n,M}T_{n,M})^r\to \Ebb Z^r=r!\), where \(Z\sim\Exp(1)\).
\end{proof}

\begin{lemma}[Uniform subcritical occupancy lower tail]\label{lem:occupancy-lower-tail}
Let \(N\ge2\) and \(m\ge N\).  If
\[
        \Lambda_{N,m}:=N\left(1-\frac1N\right)^m,
\]
then
\[
        \frac{N!S(m,N)}{N^m}
        \le \exp\{-\Lambda_{N,m}\}.
\]
Consequently, in any range where \(N\le m\le 2N\log N\) and
\(\Lambda_{N,m}\to\infty\), the onto probability is at most
\(\exp\{-c\Lambda_{N,m}\}\) for any fixed \(0<c<1\), for all sufficiently large \(N\).

Moreover, if \(M\le n\log n\), then
\[
        (n-1)\left(1-\frac1{n-1}\right)^{M-1}
        \asymp
        n\left(1-\frac1n\right)^M ,
\]
and hence the same exponential lower-tail bound may be applied to the
\((n-1,M-1)\) occupancy probability appearing in the flux formula.
\end{lemma}

\begin{proof}
Let \(C_i\) be the number of balls in box \(i\) after \(m\) iid draws from
\([N]\), and set \(Y_i=\1_{\{C_i>0\}}\).  The multinomial occupancy vector
\((C_1,\ldots,C_N)\) is negatively associated; this is the standard
negative-association property of multinomial occupancy counts, obtainable
by conditioning independent Poisson counts on their total
\cite{DubhashiRanjan}.  Since the functions \(c\mapsto \1_{\{c>0\}}\) are
coordinatewise increasing and depend on disjoint coordinates, the variables
\(Y_1,\ldots,Y_N\) are also negatively associated.  Therefore
\[
        \Pbb(Y_1=\cdots=Y_N=1)
        \le \prod_{i=1}^N \Pbb(Y_i=1)
        =
        \left(1-\left(1-\frac1N\right)^m\right)^N
        \le \exp\{-\Lambda_{N,m}\}.
\]
The left-hand side is exactly \(N!S(m,N)/N^m\).

For the comparison of the two missing-color means, write
\[
\frac{(n-1)(1-\frac1{n-1})^{M-1}}
     {n(1-\frac1n)^M}
=
\frac{n-1}{n}
\left(\frac{1-\frac1{n-1}}{1-\frac1n}\right)^{M-1}
\left(1-\frac1n\right)^{-1}.
\]
The logarithm of the middle factor is
\[
        (M-1)\log\left(1-\frac{1}{(n-1)^2}\right)
        =O(M/n^2)=O((\log n)/n),
\]
because \(M\le n\log n\).  Thus the ratio is bounded above and below by
absolute positive constants.
\end{proof}

\begin{lemma}[Fresh-block minorization for sliding windows]\label{lem:fresh-block-minorization}
In the sliding-window chain, let \(b=b_n\) satisfy
\[
        M=o(b),\qquad b\mu_{n,M}\to0,
\]
and assume the one-block estimate
\[
        \Pbb_{\rm stat}\left(\sum_{t=1}^b E_t\ge1\right)
        =b\mu_{n,M}(1+o(1)).
\]
Then, uniformly over the starting window \(w\in[n]^M\),
\[
        \Pbb_w(H_A\le M+b+1)
        \ge b\mu_{n,M}(1+o(1)).
\]
In particular, for all sufficiently large \(n\),
\[
        \inf_w\Pbb_w(H_A\le M+b+1)\ge \frac12 b\mu_{n,M}.
\]
\end{lemma}

\begin{proof}
Start the chain from an arbitrary deterministic window \(w\) at time \(0\),
and let \(X_1,X_2,\ldots\) be the fresh subsequent symbols.  For
\(M+1\le t\le M+b\), the entry event \(E_t\) depends only on the fresh
symbols \(X_{t-M},X_{t-M+1},\ldots,X_t\).  Hence the process
\[
        (E_{M+1},E_{M+2},\ldots,E_{M+b})
\]
has exactly the same distribution as
\[
        (E_1,E_2,\ldots,E_b)
\]
under the stationary two-sided iid construction.  If any of these entry
events occurs, then the target \(A_{n,M}\) is hit.  Therefore
\[
        \Pbb_w(H_A\le M+b)
        \ge
        \Pbb_{\rm stat}\left(\sum_{t=1}^b E_t\ge1\right)
        =b\mu_{n,M}(1+o(1)),
\]
uniformly in \(w\).  Enlarging \(M+b\) to \(M+b+1\) gives the displayed
form.
\end{proof}

\begin{remark}[Why the theorem uses flux]
If visits to $A_{n,M}$ last for a typical clump length $\ell_{n,M}$, then one expects
\[
        \mu_{n,M}\asymp \frac{\pi_{n,M}}{\ell_{n,M}}.
\]
For sliding windows below the coupon-collector threshold, onto windows may persist for many consecutive shifts once they appear, but new entries are much rarer than individual onto windows.  The first hitting time is controlled by the entry rate.
\end{remark}

\section{Declumping in the subcritical window regime}

The remaining input for Theorem~\ref{thm:finite-dependence} is the short-return estimate
\[
        \theta_{n,M}:=\sum_{u=1}^{M}\Pbb(E_u=1\mid E_0=1)\to0.
\]
We prove it by reducing a possible second entry to an occupancy problem for the last occurrences in the window that caused the first entry.

Throughout this section put
\[
        m:=M-1,\qquad N:=n-1,
        \qquad
        \lambda_{n,M}:=n\left(1-\frac1n\right)^M.
\]
Under the standing subcritical assumption
\[
        n\le M\le n\log n,
        \qquad
        \lambda_{n,M}\to\infty,
\]
one has a rare window.  The fixed-sum declumping estimate below is proved under the quantitative separation
\[
        \frac{\lambda_{n,M}}{(\log M)^2}\to\infty.
\]
This still includes every fixed integer scale \(M=\floor{\alpha n\log n}\), \(0<\alpha<1\), and the endpoint \(M=n\).

\begin{lemma}[Conditional two-entry reduction]\label{lem:two-entry-reduction}
Condition on \(E_0=1\), and let \(a=X_0\) be the color appended at the entry time.  Then the middle block
\[
        U=(X_{-M+1},\ldots,X_{-1})
\]
is a uniform onto word of length \(m=M-1\) over the alphabet \([n]\setminus\{a\}\).  For \(1\le u\le m\), define
\[
        R_u:=\#\bigl\{c\in[n]\setminus\{a\}:
        \text{all occurrences of }c\text{ in }U\text{ lie among }U_1,\ldots,U_u\bigr\}.
\]
Then, conditional on \(U\) and \(E_0\),
\begin{equation}\label{eq:entry-pair-bound}
        \Pbb(E_u=1\mid U,E_0)
        \le
        \frac{R_u}{n}\left(1-\frac1n\right)^{u-1}
        \exp\left\{-(R_u-1)_+\left(1-\frac1{n-1}\right)^{u-1}\right\}.
\end{equation}
Also
\begin{equation}\label{eq:endpoint-pair-bound}
        \Pbb(E_M=1\mid E_0)=\mu_{n,M}.
\end{equation}
\end{lemma}

\begin{proof}
The assertion about \(U\) follows directly from the exact entry characterization in Theorem~\ref{thm:flux}: at an entry, the common middle block is onto over the \(n-1\) colors different from the newly appended color, and contains no copy of that new color.

Fix \(1\le u\le m\).  At time \(u\), the middle block for the event \(E_u\) is
\[
        (X_{u-M+1},\ldots,X_{-1},X_0,X_1,\ldots,X_{u-1}).
\]
If \(E_u=1\), its missing color cannot be \(a\), because \(X_0=a\) belongs to this middle block.  Hence the missing color is some \(b\in[n]\setminus\{a\}\).  Since \(U\) originally contained \(b\), all occurrences of \(b\) in \(U\) must have been shifted out by time \(u\); thus \(b\) is one of the \(R_u\) colors counted above.  The appended symbol must satisfy \(X_u=b\).  In addition, the future symbols \(X_1,\ldots,X_{u-1}\) must avoid \(b\), and they must contain every other color among the \(R_u-1\) colors whose old occurrences have already shifted out.

For a fixed choice of \(b\), the probability that \(X_u=b\) is \(1/n\).  The probability that \(X_1,\ldots,X_{u-1}\) avoid \(b\) is \((1-1/n)^{u-1}\).  Conditional on this avoidance, the symbols \(X_1,\ldots,X_{u-1}\) are iid on an alphabet of size \(n-1\).  The indicators that specified colors appear are negatively associated in the usual balls-into-boxes model; equivalently, the probability that all \(r\) specified colors appear is at most the product of their marginal probabilities.  Therefore the probability that the \(R_u-1\) required colors are all present is at most
\[
        \left(1-\left(1-\frac1{n-1}\right)^{u-1}\right)^{(R_u-1)_+}
        \le
        \exp\left\{-(R_u-1)_+\left(1-\frac1{n-1}\right)^{u-1}\right\}.
\]
Summing over the at most \(R_u\) possible choices of \(b\) gives \eqref{eq:entry-pair-bound}.

For \(u=M\), the event \(E_M\) uses the fresh middle block \((X_1,\ldots,X_{M-1})\), the appended symbol \(X_M\), and the departing symbol \(X_0=a\).  Conditional on \(E_0\), these fresh symbols are independent of the past.  The event \(E_M\) occurs exactly when, for some \(b\in[n]\setminus\{a\}\), the fresh middle block contains every color in \([n]\setminus\{b\}\) and contains no \(b\), and \(X_M=b\).  Hence
\[
        \Pbb(E_M=1\mid E_0)
        =(n-1)\frac{(n-1)!S(M-1,n-1)}{n^{M-1}}\frac1n
        =\mu_{n,M},
\]
by Theorem~\ref{thm:flux}.
\end{proof}

We next isolate the only conditioning estimate needed below.  It is deliberately stated with a polynomial loss, since no sharp de-Poissonization is required.

\begin{lemma}[Polynomial conditioning cost]\label{lem:poly-conditioning-cost}
Let \(N\ge2\), \(m\ge N+1\), and \(1\le u\le m-1\), with \(m/N\le 2\log N\).  Let \(\tau>0\) be determined by
\[
        \frac{\tau}{1-e^{-\tau}}=\frac{m}{N}.
\]
For each color \(c\), let \((A_c,B_c)\) be independent split Poisson variables with means \(\tau u/m\) and \(\tau(m-u)/m\), conditioned on \(A_c+B_c\ge1\).  Write \(\Pbb_*\) for this product law and set
\[
        \mathcal C_u:=\left\{
        \sum_{c=1}^N A_c=u,
        \quad
        \sum_{c=1}^N(A_c+B_c)=m
        \right\}.
\]
There is an absolute constant \(C<\infty\) such that, uniformly in \(N,m,u\) in this range,
\[
        \Pbb_*(\mathcal C_u)\ge m^{-C}.
\]
Consequently, for every nonnegative functional \(H\) of the split occupancy configuration,
\[
        \Ebb_{\mathrm{onto}}H
        =\Ebb_*(H\mid \mathcal C_u)
        \le m^C\Ebb_*H.
\]
\end{lemma}

\begin{proof}
We first justify the displayed identity between the uniform onto law and the
conditioned split Poisson law.  If a uniform onto word of length \(m\) over
\([N]\) is split after \(u\) positions, then a split occupancy configuration
\((a_c,b_c)_{c=1}^N\), with \(a_c+b_c\ge1\), \(\sum_c a_c=u\), and
\(\sum_c(a_c+b_c)=m\), has probability proportional to
\[
        \frac{u!}{\prod_c a_c!}\frac{(m-u)!}{\prod_c b_c!}.
\]
Under \(\Pbb_*\), conditioned on \(\mathcal C_u\), the same configuration has
probability proportional to
\[
        \prod_{c=1}^N
        \frac{(\tau u/m)^{a_c}(\tau(m-u)/m)^{b_c}}{a_c!b_c!}.
\]
On \(\mathcal C_u\), the powers of \(\tau u/m\) and \(\tau(m-u)/m\) are
constant, so the two conditional laws are identical.  Thus
\(\Ebb_{\mathrm{onto}}H=\Ebb_*(H\mid\mathcal C_u)\) for every nonnegative
functional \(H\) of the split occupancy configuration.

Let \(D_c=A_c+B_c\).  Under \(\Pbb_*\), the variables \(D_c\) are iid
zero-truncated Poisson variables with parameter \(\tau\), and
\[
        \Ebb_*D_c=\frac{\tau}{1-e^{-\tau}}=\frac mN.
\]
Thus \(S_N:=\sum_{c=1}^N D_c\) has mean \(m\).  Conditional on the
total \(S_N=k\), each of the \(k\) points falls in the first block with
probability \(u/m\), independently of its color-count allocation.  Hence
\[
        \Pbb_*(\mathcal C_u)
        =\Pbb_*(S_N=m)\,
          \Pbb\left(\Bin\left(m,\frac um\right)=u\right).
\]
By Stirling's formula, uniformly for \(1\le u\le m-1\),
\[
        \Pbb\left(\Bin\left(m,\frac um\right)=u\right)
        \ge c m^{-1/2}
\]
for an absolute constant \(c>0\).

It remains to give a polynomial lower bound for \(\Pbb_*(S_N=m)\).  Let
\[
        \sigma_\tau^2:=\Var_*(D_1),
        \qquad
        B:=N\sigma_\tau^2.
\]
Since \(m/N\le2\log N\), the defining relation
\(\tau/(1-e^{-\tau})=m/N\) implies \(0<\tau\le C\log m\).  If
\(B\) is sufficiently large, Lemma~\ref{lem:uniform-ztp-llt} gives
\[
        \Pbb_*(S_N=m)=(2\pi B)^{-1/2}(1+o(1)).
\]
Moreover, the same lemma records \(\sigma_\tau^2\asymp\tau\) in this
range, and therefore
\[
        B=N\sigma_\tau^2\le C N\tau\le C N\log N\le C m\log m.
\]
Thus \(\Pbb_*(S_N=m)\ge m^{-C}\) whenever \(B\) is large.

For the remaining bounded-\(B\) regime, \(\tau=O(1/N)\).  Indeed, for small \(\tau\) the
zero-truncated Poisson law satisfies \(\sigma_\tau^2\asymp\tau\), while
for \(\tau\) bounded away from \(0\) the variance is bounded below by an
absolute positive constant.  Hence
\[
        k:=m-N
        =N\left(\frac{\tau}{1-e^{-\tau}}-1\right)
        =O(N\tau)=O(1).
\]
In this case the event that exactly \(k\) of the \(D_c\)'s equal \(2\) and
all remaining \(D_c\)'s equal \(1\) implies \(S_N=m\).  Its probability is
\[
        \binom Nk p_2^k p_1^{N-k},
        \qquad
        p_j:=\Pbb_*(D_1=j).
\]
For \(\tau=O(1/N)\),
\[
        p_1=1-\frac{\tau}{2}+O(\tau^2),
        \qquad
        p_2=\frac{\tau}{2}+O(\tau^2),
\]
and since \(k=O(1)\), the displayed probability is bounded below by a
negative power of \(m\).  Therefore \(\Pbb_*(S_N=m)\ge m^{-C}\) in all
cases.  Combining this with the binomial factor gives
\(\Pbb_*(\mathcal C_u)\ge m^{-C}\), after increasing \(C\).  The displayed
conditioning inequality follows.
\end{proof}

We also isolate the elementary saddle-point comparison used in the weighted
last-occurrence estimate.

\begin{lemma}[Last-occurrence scale comparison]\label{lem:last-occurrence-scale}
Let \(N=n-1\), \(m=M-1\), and assume
\[
        n\le M\le n\log n,
        \qquad
        M>n.
\]
Let \(\tau>0\) be defined by
\[
        \frac{\tau}{1-e^{-\tau}}=\frac mN=:\rho .
\]
For \(1\le u\le m-1\), put
\[
        q_u:=\left(1-\frac1N\right)^{u-1},
        \qquad
        y_u:=1-e^{-u/n},
\]
and
\[
        \Lambda_u
        :=N\frac{e^{-\tau(m-u)/m}(1-e^{-\tau u/m})}{1-e^{-\tau}}.
\]
Then, uniformly for \(1\le u\le m-1\),
\[
        q_u\Lambda_u\asymp
        \lambda_{n,M}y_u,
        \qquad
        \lambda_{n,M}:=n\left(1-\frac1n\right)^M.
\]
\end{lemma}

\begin{proof}
Set \(v=u/N\).  Since \(1\le u\le m-1\), we have \(0<v<\rho\).  Also
\(\tau/\rho=1-e^{-\tau}\).  We first compare the main continuous factors.
Using the saddle-point relation,
\[
\begin{aligned}
&\frac{e^{-v}e^{-\tau(1-v/\rho)}(1-e^{-\tau v/\rho})/(1-e^{-\tau})}
        {e^{-\rho}(1-e^{-v})}  \\
&\qquad
=\exp\{(\rho-\tau)+v(\tau/\rho-1)\}
  \frac{1-e^{-(1-e^{-\tau})v}}
       {(1-e^{-\tau})(1-e^{-v})}  \\
&\qquad
=\exp\{e^{-\tau}(\rho-v)\}
  \frac{1-e^{-(1-e^{-\tau})v}}
       {(1-e^{-\tau})(1-e^{-v})}.
\end{aligned}
\]
The exponential factor is bounded above and below by absolute constants,
because \(0<v<\rho\) and
\[
        \rho e^{-\tau}=\frac{\tau}{e^\tau-1}\le1.
\]
For the second factor, write \(a=1-e^{-\tau}\).  Concavity of
\(a\mapsto1-e^{-av}\) gives
\[
        1-e^{-av}\ge a(1-e^{-v}),
\]
so the factor is at least \(1\).  For the upper bound, split into two cases.
If \(\tau\le1\), then \(\rho=\tau/(1-e^{-\tau})\le C\), so
\[
        \frac{1-e^{-av}}{a(1-e^{-v})}
        \le \frac{v}{1-e^{-v}}
        \le C
        \qquad(0<v<\rho).
\]
If \(\tau\ge1\), then \(a\ge1-e^{-1}\), and the factor is at most
\(1/a\le C\).  Hence
\[
        e^{-v}\,\frac{e^{-\tau(1-v/\rho)}(1-e^{-\tau v/\rho})}{1-e^{-\tau}}
        \asymp e^{-\rho}(1-e^{-v})
\]
uniformly for \(0<v\le\rho\).  Multiplying by \(N\) gives
\[
        e^{-v}\Lambda_u\asymp N e^{-\rho}(1-e^{-v}).
\]
Finally,
\[
        q_u\asymp e^{-u/N}=e^{-v},
        \qquad
        1-e^{-v}\asymp 1-e^{-u/n}=y_u,
\]
uniformly in the present range.  Also
\[
        N e^{-m/N}\asymp n\left(1-\frac1n\right)^M=\lambda_{n,M},
\]
because \(M\le n\log n\) and
\[
        \log\frac{N e^{-m/N}}{n(1-1/n)^M}=O\left(\frac{\log n}{n}\right).
\]
Combining these estimates proves the claim.
\end{proof}

The next lemma is the only occupancy estimate needed for declumping.  It is stated in exactly the weighted form supplied by Lemma~\ref{lem:two-entry-reduction}.

\begin{lemma}[Last-occurrence occupancy bound]\label{lem:last-occurrence-bound}
Let \(U\) be a uniform onto word of length \(m=M-1\) over an alphabet of size \(N=n-1\), and let \(R_u\) be the number of letters whose last occurrence in \(U\) is at most \(u\).  Assume
\[
        n\le M\le n\log n,
        \qquad
        \frac{\lambda_{n,M}}{(\log M)^2}\longrightarrow\infty,
        \qquad
        \lambda_{n,M}:=n\left(1-\frac1n\right)^M .
\]
Then
\begin{equation}\label{eq:last-occurrence-sum}
\sum_{u=1}^{M-1}
\Ebb\left[
        \frac{R_u}{n}\left(1-\frac1n\right)^{u-1}
        \exp\left\{-(R_u-1)_+\left(1-\frac1{n-1}\right)^{u-1}\right\}
\right]
        =O\left(\frac{\log M}{\lambda_{n,M}}\right).
\end{equation}
In particular the sum is \(o(1)\).
\end{lemma}

\begin{proof}
Set \(\lambda=\lambda_{n,M}\) for short.  We first remove the degenerate endpoint.  If \(M=n\), then \(m=N\), and the uniform onto word \(U\) is a permutation of the \(N\) letters.  Hence \(R_u=u\) for \(1\le u\le N\).  Since \(q_u:=(1-1/(n-1))^{u-1}\ge e^{-2}\) throughout this range, the summand is bounded by
\[
        \frac{Cu}{n}\exp\{-cu\}.
\]
Summing over \(u\ge1\) gives \(O(n^{-1})\), which is
\(O((\log M)/\lambda)\) because \(\lambda\asymp n\) when \(M=n\).

Assume now that \(M>n\), so that \(m>N\).  Put
\[
        q_u:=\left(1-\frac1{n-1}\right)^{u-1},
        \qquad
        y_u:=1-e^{-u/n}.
\]
The endpoint \(u=m\) is harmless and will be separated off.  In this case
\(R_m=N\) deterministically, so the corresponding summand is at most
\[
        C\left(1-\frac1n\right)^M
        \exp\left\{-cN\left(1-\frac1N\right)^m\right\}
        \le C\exp\{-c'\lambda_{n,M}\},
\]
which is \(o(\lambda_{n,M}^{-1})\) under
\(\lambda_{n,M}/(\log M)^2\to\infty\).  We may therefore restrict the
Poissonization argument below to \(1\le u\le m-1\).

We use a Poisson comparison only to obtain binomial Laplace estimates, and keep the cost of returning to the fixed-length onto model explicit.

Fix \(1\le u\le m-1\).  Let \(\tau=\tau_{N,m}>0\) be the solution of
\[
        \frac{N\tau}{1-e^{-\tau}}=m.
\]
For each letter \(c\), let \(A_c\) and \(B_c\) be independent Poisson variables with means \(\tau u/m\) and \(\tau(m-u)/m\), respectively, conditioned on
\[
        A_c+B_c\ge1\qquad(1\le c\le N).
\]
In this product model the events \(\{B_c=0\}\) are independent.  If
\[
        R_u^*:=\sum_{c=1}^N\1_{\{B_c=0\}},
        \qquad
        \Lambda_u:=\Ebb_*R_u^*,
\]
then, by Lemma~\ref{lem:last-occurrence-scale},
\begin{equation}\label{eq:lambda-u-comparison}
        q_u\Lambda_u\asymp \lambda y_u
\end{equation}
uniformly for \(1\le u\le m-1\).

The fixed-length uniform onto word is obtained from this product model by additionally conditioning on
\[
        \mathcal C_u=\left\{
        \sum_{c=1}^N A_c=u,
        \quad
        \sum_{c=1}^N(A_c+B_c)=m
        \right\}.
\]
By Lemma~\ref{lem:poly-conditioning-cost}, for every nonnegative functional \(H\) of the split occupancy configuration,
\begin{equation}\label{eq:poly-conditioning-cost}
        \Ebb_{\mathrm{onto}}H\le C M^C\Ebb_*H .
\end{equation}
This is the only de-Poissonization input used below.  The polynomial factor will be applied only in a range where it is absorbed by an exponential Laplace penalty.

Let
\[
        \eta_u:=q_u\Lambda_u.
\]
We split the sum according to the size of \(\eta_u\).  Choose a constant \(A\) large enough later.

First consider the indices for which \(\eta_u\le A\log M\).  For every integer \(r\ge0\),
\[
        q_u r\exp\{ -q_u(r-1)_+\}\le C.
\]
Also
\[
        \left(1-\frac1n\right)^{u-1}\le Cq_u
        \qquad (1\le u\le m-1),
\]
because \(M\le n\log n\).  Therefore the corresponding part of the sum in \eqref{eq:last-occurrence-sum} is at most
\[
        \frac{C}{n}\#\{u:\eta_u\le A\log M\}.
\]
By \eqref{eq:lambda-u-comparison}, \(\eta_u\asymp\lambda y_u\).  Since \(y_u=1-e^{-u/n}\asymp u/n\) for \(u\le n\) and is bounded below by an absolute constant for \(u>n\), the number of such indices is at most
\[
        C n\frac{\log M}{\lambda}.
\]
Hence the small-\(\eta_u\) contribution is
\begin{equation}\label{eq:small-eta-contribution}
        O\left(\frac{\log M}{\lambda}\right).
\end{equation}

It remains to treat the indices with \(\eta_u>A\log M\).  Since \(R_u^*\) is binomial with mean \(\Lambda_u\), the binomial Laplace transform gives
\begin{equation}\label{eq:binomial-laplace}
\begin{aligned}
        \Ebb_*\left[q_uR_u^*e^{-q_u(R_u^*-1)_+}\right]
        &\le C\Ebb_*\left[q_uR_u^*e^{-q_uR_u^*/2}\right]  \\
        &\le C\eta_u e^{-c\eta_u}.
\end{aligned}
\end{equation}
For example, if \(R_u^*\sim\Bin(N,p_u)\), then
\[
        \Ebb_*\left[q_uR_u^*e^{-q_uR_u^*/2}\right]
        =q_u N p_u e^{-q_u/2}(1-p_u+p_ue^{-q_u/2})^{N-1}
        \le \eta_u e^{-c q_u N p_u}
        =\eta_u e^{-c\eta_u}.
\]
Combining \eqref{eq:poly-conditioning-cost} and \eqref{eq:binomial-laplace}, and choosing \(A\) sufficiently large, gives
\begin{equation}\label{eq:large-eta-bound}
        \Ebb_{\mathrm{onto}}\left[q_uR_ue^{-q_u(R_u-1)_+}\right]
        \le e^{-c'\eta_u}
        \qquad(\eta_u>A\log M).
\end{equation}
Indeed, the factor \(M^C\eta_u\) is absorbed into \(e^{-c\eta_u/2}\) throughout this range.

Using again \((1-1/n)^{u-1}\le Cq_u\), the large-\(\eta_u\) contribution to \eqref{eq:last-occurrence-sum} is at most
\[
        \frac{C}{n}\sum_{u=1}^{m-1}e^{-c'\eta_u}.
\]
By \eqref{eq:lambda-u-comparison}, \(\eta_u\ge c\lambda y_u\).  For \(u\le n\), \(y_u\ge cu/n\), so
\[
        \frac1n\sum_{u=1}^{n}e^{-c'\eta_u}
        \le
        \frac1n\sum_{u=1}^{n}\exp\left\{-c''\lambda\frac{u}{n}\right\}
        =O\left(\frac1{\lambda}\right).
\]
For \(u>n\), \(y_u\ge c\), and hence
\[
        \frac1n\sum_{u=n+1}^{M-1}e^{-c'\eta_u}
        \le C\log n\,e^{-c''\lambda}
        =o\left(\frac1{\lambda}\right),
\]
because \(\lambda/(\log M)^2\to\infty\).  Thus the large-\(\eta_u\) contribution is
\begin{equation}\label{eq:large-eta-contribution}
        O\left(\frac1{\lambda}\right).
\end{equation}
Combining \eqref{eq:small-eta-contribution} and \eqref{eq:large-eta-contribution} proves \eqref{eq:last-occurrence-sum}.
\end{proof}

\begin{lemma}[Subcritical short-return estimate]\label{lem:short-return}
Let \(M=M_n\) satisfy
\[
        n\le M\le n\log n,
        \qquad
        \frac{\lambda_{n,M}}{(\log M)^2}\to\infty.
\]
Then
\[
        \theta_{n,M}:=\sum_{u=1}^{M}\Pbb(E_u=1\mid E_0=1)\to0.
\]
More quantitatively,
\[
        \theta_{n,M}=O\left(\frac{\log M}{\lambda_{n,M}}\right)=o(1).
\]
\end{lemma}

\begin{proof}
For \(1\le u\le M-1\), Lemma~\ref{lem:two-entry-reduction} bounds \(\Pbb(E_u=1\mid U,E_0)\) by the weighted last-occurrence expression appearing in Lemma~\ref{lem:last-occurrence-bound}.  Taking expectation over the conditional middle block \(U\) and summing over \(u\le M-1\) gives
\[
        \sum_{u=1}^{M-1}\Pbb(E_u=1\mid E_0=1)
        =O\left(\frac{\log M}{\lambda_{n,M}}\right).
\]
The endpoint \(u=M\) contributes \(O(\mu_{n,M})\) by \eqref{eq:endpoint-pair-bound}.  By Theorem~\ref{thm:flux} and Lemma~\ref{lem:occupancy-lower-tail},
\[
        \mu_{n,M}
        =
        \left(1-\frac1n\right)^M
        \frac{(n-1)!S(M-1,n-1)}{(n-1)^{M-1}}
        \le \exp\{-c\lambda_{n,M}\}
\]
for some absolute \(c>0\), in the present range.  This endpoint contribution
is absorbed by the preceding \(O((\log M)/\lambda_{n,M})\) term.  The asserted
estimate follows.
\end{proof}

\begin{remark}[Mechanism of the declumping estimate]
The proof does not rely on an informal renewal picture.  A second entry is reduced to the following finite combinatorial requirement: among the colors whose old last occurrences have already shifted out, one color must be selected as the new missing color, all the others must be repaired by fresh draws, and the selected missing color must be avoided until the final repairing draw.  The weighted last-occurrence estimate now keeps the fixed-length conditioning cost explicit and absorbs it before summing over overlaps.  The price is the mild quantitative separation \(\lambda_{n,M}\gg(\log M)^2\), which covers all fixed \(M=\floor{\alpha n\log n}\), \(0<\alpha<1\), and the endpoint \(M=n\).
\end{remark}

\section{Subcritical rare-window law}

We now combine the exact flux and the finite-dependence theorem.

\begin{theorem}[Subcritical expiring coupon collector]\label{thm:subcritical}
Let $M=M_n$ satisfy
\[
        n\le M\le n\log n,
        \qquad
        \frac{\lambda_{n,M}}{(\log M)^2}\to\infty,
        \qquad
        \lambda_{n,M}:=n\left(1-\frac1n\right)^M.
\]
Let $T_{n,M}$ be the first time that the expiring coupon collector has all $n$ types simultaneously active.  Then
\[
        \mu_{n,M}T_{n,M}\Rightarrow\Exp(1),
\]
where
\[
        \mu_{n,M}
        =\frac{(n-1)(n-1)!S(M-1,n-1)}{n^M}.
\]
Moreover, for every fixed $r\ge1$,
\[
        \Ebb(\mu_{n,M}T_{n,M})^r\to r!,
\]
and in particular
\[
        \Ebb T_{n,M}\sim\frac1{\mu_{n,M}},
        \qquad
        \Var(T_{n,M})\sim\frac1{\mu_{n,M}^2}.
\]
\end{theorem}

\begin{proof}
By Proposition~\ref{prop:mass},
\[
        \pi_{n,M}=\Pbb(W_t\in A_{n,M}).
\]
Under the theorem's hypotheses,
\[
        \lambda_{n,M}=n\left(1-\frac1n\right)^M\to\infty.
\]
Lemma~\ref{lem:occupancy-lower-tail} gives, uniformly in this range,
\[
        \pi_{n,M}\le \exp\{-\lambda_{n,M}\}\to0.
\]
The same lemma, applied to the exact flux formula of
Theorem~\ref{thm:flux}, gives
\[
        \mu_{n,M}
        \le
        \left(1-\frac1n\right)^M
        \exp\{-c\lambda_{n,M}\}
        \le \exp\{-c'\lambda_{n,M}\}.
\]
Since \(\lambda_{n,M}\gg(\log M)^2\), this implies
\[
        \mu_{n,M}\to0,
        \qquad
        M\mu_{n,M}\to0.
\]
Lemma~\ref{lem:short-return} gives \(\theta_{n,M}\to0\).  Therefore
Theorem~\ref{thm:finite-dependence} applies.

For the moment statement, choose the block length \(b\) from the proof of
Theorem~\ref{thm:finite-dependence}.  The one-block estimate proved there,
combined with Lemma~\ref{lem:fresh-block-minorization}, supplies the required
uniform lower bound.  Hence the moment part of
Theorem~\ref{thm:finite-dependence} applies.
\end{proof}

\section{Fixed-\texorpdfstring{$\alpha$}{alpha} scale}

The regime emphasized in \cite{MSEAnswer} is
\[
        M=\floor{\alpha n\log n},
        \qquad 0<\alpha<1.
\]
The exact flux gives the canonical finite-\(n\) normalization.  Its elementary leading logarithmic form is as follows.

\begin{lemma}[Fixed-$\alpha$ flux scale]\label{lem:fixed-alpha-flux}
If \(M=\floor{\alpha n\log n}\) with fixed \(0<\alpha<1\), then
\[
        \log \mu_{n,M}=-n^{1-\alpha}+o(n^{1-\alpha}).
\]
If, moreover, \(\alpha>1/2\), then the sharper ratio asymptotic holds:
\[
        \mu_{n,M}\sim n^{-\alpha}\exp(-n^{1-\alpha}).
\]
\end{lemma}

\begin{proof}
By Theorem~\ref{thm:flux},
\[
        \mu_{n,M}
        =\left(1-\frac1n\right)^M
          \rho_{n-1,M-1},
\]
where
\[
        \rho_{N,m}:=\frac{N!S(m,N)}{N^m}
\]
is the probability that an iid word of length \(m\) over \(N\) symbols is onto.
We need the lower-tail estimate for
\[
        m=\alpha N\log N+O(\log N),
        \qquad 0<\alpha<1.
\]
Let \(Z_N\) be the number of missing symbols.  Then
\[
        \Lambda_N:=\Ebb Z_N
        =N\left(1-\frac1N\right)^m
        =N^{1-\alpha}(1+o(1)).
\]
The classical Erd\H{o}s--R\'enyi coupon-collector lower-tail theorem
\cite{ErdosRenyi,Feller} gives
\[
        \log \rho_{N,m}=-\Lambda_N+o(\Lambda_N)
        =-N^{1-\alpha}+o(N^{1-\alpha}).
\]
For completeness, we recall how this follows from inclusion-exclusion.  Since
\[
        \rho_{N,m}
        =\sum_{k=0}^{N}(-1)^k\binom Nk\left(1-\frac{k}{N}\right)^m,
\]
the \(k\)-th term in the relevant range satisfies
\[
        \binom Nk\left(1-\frac{k}{N}\right)^m
        =\frac{\Lambda_N^k}{k!}
          \exp\left\{O\left(\frac{k^2}{N}+\frac{m k^2}{N^2}\right)\right\}
\]
uniformly for \(k=O(\Lambda_N)\).  Because
\(\Lambda_N^2/N+m\,\Lambda_N^2/N^2=o(\Lambda_N)\), Bonferroni truncation at
orders \((1\pm\varepsilon)\Lambda_N\), followed by \(\varepsilon\downarrow0\),
gives the displayed logarithmic estimate.  This is the standard
inclusion-exclusion proof of the coupon-collector lower tail.

Since
\[
        \log\left(1-\frac1n\right)^M=-\alpha\log n+o(1)=o(n^{1-\alpha}),
\]
the logarithmic estimate for \(\mu_{n,M}\) follows.

For \(\alpha>1/2\), the same inclusion-exclusion expansion is accurate on the
multiplicative scale.  Indeed, the first cluster correction is
\[
        \binom N2\left[
        \left(1-\frac2N\right)^m
        -
        \left(1-\frac1N\right)^{2m}\right]
        =O\left(N^{1-2\alpha}\log N\right)=o(1).
\]
More generally, writing \(\delta_N:=m/N^2=O(\log N/N)\), the connected
\(r\)-point cluster correction in the inclusion-exclusion logarithm is
\[
        O_r\left(\Lambda_N^r\delta_N^{r-1}\right)
        =O_r\left(N^{1-r\alpha}(\log N)^{r-1}\right)=o(1)
\]
for each fixed \(r\ge2\).  The Bonferroni remainder is then controlled by
letting the truncation order tend slowly to infinity, so
\(\log\rho_{N,m}=-\Lambda_N+o(1)\).  Hence
\[
        \rho_{N,m}=\exp(-\Lambda_N)(1+o(1))
        =\exp(-N^{1-\alpha})(1+o(1)).
\]
Together with \((1-1/n)^M\sim n^{-\alpha}\), this gives the ratio asymptotic.
\end{proof}

\begin{corollary}[Fixed-$\alpha$ expiring collector]\label{cor:fixed-alpha}
Let \(0<\alpha<1\) be fixed and put \(M=\floor{\alpha n\log n}\).  For all sufficiently large \(n\),
\[
        \mu_{n,M}T_{n,M}\Rightarrow\Exp(1),
\]
where
\[
        \mu_{n,M}=\frac{(n-1)(n-1)!S(M-1,n-1)}{n^M}.
\]
Consequently, for every fixed \(r\ge1\),
\[
        \Ebb(\mu_{n,M}T_{n,M})^r\to r!,
\]
and
\[
        \log T_{n,M}=n^{1-\alpha}+o_{\mathbb P}(n^{1-\alpha}),
        \qquad
        \log \Ebb T_{n,M}=n^{1-\alpha}+o(n^{1-\alpha}).
\]
If \(\alpha>1/2\), then the sharper normalization is
\[
        n^{-\alpha}e^{-n^{1-\alpha}}T_{n,M}\Rightarrow\Exp(1),
\]
and in particular
\[
        \Ebb T_{n,M}\sim n^\alpha e^{n^{1-\alpha}},
        \qquad
        \Var(T_{n,M})\sim n^{2\alpha}e^{2n^{1-\alpha}}.
\]
\end{corollary}

\begin{proof}
The exact-flux limit and moment convergence are Theorem~\ref{thm:subcritical}.  The logarithmic estimates and, for \(\alpha>1/2\), the sharper normalization follow from Lemma~\ref{lem:fixed-alpha-flux}.
\end{proof}

\begin{remark}[Comparison with the StackExchange answer]
Joriki's Math StackExchange answer predicts the leading scale \(n^\alpha\exp(n^{1-\alpha})\) for \(M=\floor{\alpha n\log n}\), \(0<\alpha<1\), and gives a correction analysis for overlapping windows \cite{MSEAnswer}.  Corollary~\ref{cor:fixed-alpha} makes the leading logarithmic scale rigorous for all fixed \(\alpha\in(0,1)\), identifies the exact finite-\(n\) flux controlling the expectation, and gives the full exponential hitting law and variance asymptotic.  The displayed ratio asymptotic \(\Ebb T_{n,M}\sim n^\alpha e^{n^{1-\alpha}}\) is valid in the range \(\alpha>1/2\); for smaller \(\alpha\), the exact flux remains the canonical normalization and contains additional lower-order exponential terms familiar from coupon-collector lower tails.
\end{remark}

\section{Endpoint and threshold regimes}

The exact flux is useful beyond the fixed-$\alpha$ subcritical scale.

\subsection{The endpoint \texorpdfstring{$M=n$}{M=n}}

If $M=n$, completion means that the current window is a permutation of the $n$ coupon types.  Since $S(n-1,n-1)=1$,
\[
        \mu_{n,n}=\frac{(n-1)(n-1)!}{n^n}.
\]
Thus Theorem~\ref{thm:subcritical} gives
\[
        \frac{(n-1)(n-1)!}{n^n}T_{n,n}\Rightarrow\Exp(1).
\]
By Stirling's formula,
\[
        \Ebb T_{n,n}
        \sim \frac{n^n}{(n-1)(n-1)!}
        \sim \frac{e^n}{\sqrt{2\pi n}}.
\]
This is asymptotic to the reciprocal of the raw permutation-window probability $n!/n^n$; indeed
\[
        \frac{\mu_{n,n}^{-1}}{n^n/n!}
        =\frac{n!/n^n}{\mu_{n,n}}
        =\frac{n}{n-1}=1+O(n^{-1}).
\]
Thus, at the endpoint \(M=n\), permutation windows have only constant-order local clumping.

\subsection{Linear windows}\label{subsec:linear-windows}

The same exact flux gives a clean large-deviation answer when the expiration window is proportional to the number of coupon types.  Fix \(a>1\) and let
\[
        M=\lfloor an\rfloor .
\]
Then
\[
        \lambda_{n,M}=n\left(1-\frac1n\right)^M\sim ne^{-a},
\]
so Theorem~\ref{thm:subcritical} applies.  Hence
\[
        \mu_{n,M}T_{n,M}\Rightarrow\Exp(1),
        \qquad
        \Ebb T_{n,M}\sim \frac1{\mu_{n,M}}.
\]
To express the exponential rate, let \(\tau=\tau(a)>0\) be the unique solution of
\[
        \frac{\tau}{1-e^{-\tau}}=a,
\]
and set
\[
        I(a):=a\log\frac{\tau}{a}+a-\log(e^\tau-1).
\]
The saddle-point asymptotics for Stirling numbers of the second kind
\cite{TemmeStirling} give, with
\(\rho_{K,m}=K!S(m,K)/K^m\), that whenever \(m/K\to a\),
\[
        -\frac1K\log \rho_{K,m}\longrightarrow I(a).
\]
Since
\[
        \mu_{n,M}
        =\left(1-\frac1n\right)^M\rho_{n-1,M-1},
\]
and \((1-1/n)^M\to e^{-a}\), the prefactor is subexponential on the
\(n\)-scale and therefore does not affect the exponential rate.  It follows that
\[
        \frac1n\log \Ebb T_{n,\lfloor an\rfloor}\longrightarrow I(a).
\]
Thus, for every fixed \(a>1\), the expected completion time is exponential in \(n\), with explicit rate \(I(a)\).  The endpoint \(a\downarrow1\) is consistent with the permutation-window scale above, while \(I(a)\sim e^{-a}\) as \(a\to\infty\), matching the first-order behavior of the logarithmic-window regime.

\subsection{Critical windows}

Let
\[
        M=n\log n+cn+o(n),
        \qquad c\in\Rbb.
\]
Then the classical occupancy theorem gives
\[
        \pi_{n,M}\to e^{-e^{-c}}.
\]
The target is no longer rare when $c$ is fixed.  Nevertheless the exact flux has a limit on the $1/n$ scale:
\[
\begin{aligned}
        \mu_{n,M}
        &=\left(1-\frac1n\right)^M
          \frac{(n-1)!S(M-1,n-1)}{(n-1)^{M-1}} \\
        &\sim \frac{e^{-c}}{n}e^{-e^{-c}}.
\end{aligned}
\]
This is a scan-statistical regime rather than a rare-target regime: the window is onto for a positive fraction of times, but new entries into the onto set occur on the order-$n$ scale.  The first completion time from an empty start is then dominated by the ordinary coupon-collector approach to the threshold, not by a stationary rare-entry waiting time.

\subsection{Supercritical windows}

If
\[
        M=n\log n+a_n n,
        \qquad a_n\to+\infty,
\]
then a classical coupon collector completes before the first expiration with probability tending to one.  In that regime expiration is asymptotically invisible, and
\[
        \frac{T_{n,M}-n\log n}{n}
\]
has the usual Gumbel limiting behavior of the classical coupon collector.

\section{Summary of mechanisms}

The expiring model is a clean third example of the rare-entry perspective for nonmonotone coupon collectors:
\[
\begin{array}{ccl}
\text{clumsy collector} &:& \text{product stationary law and coordinate-refresh flux},\\[1mm]
\text{careless collector} &:& \text{lucky-climb high tail and thinning flux},\\[1mm]
\text{expiring collector} &:& \text{sliding-window surjection flux}.
\end{array}
\]
In the expiring model, the exact flux is
\[
        \mu_{n,M}=\frac{(n-1)(n-1)!S(M-1,n-1)}{n^M}.
\]
This identity is the main finite combinatorial input.  Once it is paired with finite-dependence declumping, the completion time has the universal rare-entry form
\[
        \mu_{n,M}T_{n,M}\Rightarrow\Exp(1).
\]
For $M=\floor{\alpha n\log n}$, $0<\alpha<1$, this yields the rigorous leading logarithmic scale
\[
        \log \Ebb T_{n,M}=n^{1-\alpha}+o(n^{1-\alpha}),
\]
and for $\alpha>1/2$ the sharper ratio asymptotic
\[
        \Ebb T_{n,M}\sim n^\alpha e^{n^{1-\alpha}}.
\]
For proportional windows $M=\lfloor an\rfloor$, $a>1$, the same flux gives
\[
        n^{-1}\log\Ebb T_{n,M}\to I(a),
\]
with the explicit saddle-point rate stated above.  Thus the original StackExchange heuristic is sharpened by an exact finite-$n$ flux, an exponential limit law, and moment asymptotics across several subcritical regimes.

\appendix

\section{Uniform local estimate for zero-truncated Poisson sums}
\label{app:ztp-local}

The following lemma is a self-contained local estimate for the
zero-truncated Poisson family.  The main text only needs the polynomial lower
bound in the final sentence of the lemma; the sharper local asymptotic is
included for completeness.  The proof is a standard Fourier-inversion argument
in the spirit of lattice local limit theorems; for background on the
classical fixed-law setting, see Petrov~\cite[Ch.~VII]{Petrov}.

\begin{lemma}[Uniform local bound for zero-truncated Poisson sums]
\label{lem:uniform-ztp-llt}
Let \(N\ge2\), \(m\ge N+1\), and let \(\tau=\tau_{N,m}>0\) be the unique
solution of
\[
        \frac{\tau}{1-e^{-\tau}}=\frac{m}{N}.
\]
Let \(D_1,\ldots,D_N\) be iid zero-truncated Poisson variables with parameter
\(\tau\):
\[
        \Pbb(D_i=k)=\frac{\tau^k}{k!(e^\tau-1)},\qquad k\ge1.
\]
Write
\[
        \mu_\tau=\Ebb D_i=\frac{\tau}{1-e^{-\tau}},
        \qquad
        \sigma_\tau^2=\Var(D_i),
        \qquad
        B=N\sigma_\tau^2.
\]
Assume \(0<\tau\le C_0\log N\). If \(B\to\infty\), then
\[
        \Pbb(D_1+\cdots+D_N=m)
        =
        \frac{1+o(1)}{\sqrt{2\pi B}},
\]
uniformly in this range. Consequently, for some constants \(c,C>0\),
depending only on \(C_0\),
\[
        \Pbb(D_1+\cdots+D_N=m)\ge c\,m^{-C}
\]
for all sufficiently large \(N\). If \(B\) is bounded, the same polynomial
lower bound still holds.
\end{lemma}

\begin{proof}
The function
\[
        \tau\mapsto \frac{\tau}{1-e^{-\tau}}
\]
is continuous and strictly increasing on \((0,\infty)\), has limit \(1\) as
\(\tau\downarrow0\), and tends to \(+\infty\) as \(\tau\to\infty\).  Hence,
for every \(m>N\), the parameter \(\tau=\tau_{N,m}\) is uniquely defined.

Let \(D=D_\tau\). Its characteristic function is
\[
        \phi_\tau(t)
        :=
        \Ebb e^{itD}
        =
        \frac{e^{\tau e^{it}}-1}{e^\tau-1}.
\]
Set
\[
        \psi_\tau(t):=e^{-i\mu_\tau t}\phi_\tau(t).
\]
Since \(m=N\mu_\tau\), Fourier inversion gives
\[
        \Pbb(D_1+\cdots+D_N=m)
        =
        \frac1{2\pi}\int_{-\pi}^{\pi}\psi_\tau(t)^N\,dt .
\]

We first record the elementary estimates used below. The first two moments are
\[
        \Ebb D=\frac{\tau}{1-e^{-\tau}},
        \qquad
        \Ebb D^2=\frac{\tau^2+\tau}{1-e^{-\tau}},
\]
and therefore
\[
        \sigma_\tau^2
        =
        \frac{\tau^2+\tau}{1-e^{-\tau}}
        -
        \left(\frac{\tau}{1-e^{-\tau}}\right)^2 .
\]
In particular,
\[
        \sigma_\tau^2\sim \frac{\tau}{2}\quad(\tau\downarrow0),
        \qquad
        \sigma_\tau^2\sim \tau\quad(\tau\to\infty).
\]
By these endpoint asymptotics and compactness on intermediate \(\tau\)-ranges,
\[
        \sigma_\tau^2\asymp \tau
\]
uniformly on \(0<\tau\le C_0\log N\).

Let
\[
        \rho_3(\tau):=\Ebb |D-\mu_\tau|^3 .
\]
We shall use the bound
\[
        \rho_3(\tau)\le C\sigma_\tau^2(1+\tau)^2 .
\]
For \(0<\tau\le1\), the distribution satisfies
\[
        \Pbb(D=1)\ge1-C\tau,
        \qquad
        \Pbb(D\ge2)=O(\tau),
        \qquad
        \Ebb[D^3\mathbf 1_{\{D\ge2\}}]=O(\tau),
\]
so \(\rho_3(\tau)=O(\tau)\), while \(\sigma_\tau^2\asymp\tau\). For
\(\tau\ge1\), the crude estimate
\[
        \rho_3(\tau)
        \le C\bigl(\Ebb D^3+\mu_\tau^3\bigr)
        \le C(1+\tau)^3
\]
is enough, since \(\sigma_\tau^2\asymp\tau\). This proves the displayed
third-moment bound.

Next we prove the characteristic-function damping estimates.  Choose
\(0<\tau_0<1<T<\infty\) and \(0<\delta_0<1\).  In the three subcases below
we may reduce \(\delta_0\) and enlarge \(T\) finitely many times; after this
is done, the final constants are fixed once and for all.  We shall prove
that there are constants \(c>0\) and \(\delta_0>0\), independent of
\(N,m,\tau\), such that, uniformly in \(\tau>0\),
\[
        |\phi_\tau(t)|\le \exp\{-c\tau t^2\},
        \qquad |t|\le\delta_0.
\]
Since \(\sigma_\tau^2\asymp\tau\) in the range used in the lemma, this is
equivalent there, up to changing \(c\), to
\[
        |\phi_\tau(t)|\le \exp\{-c\sigma_\tau^2t^2\},
        \qquad |t|\le\delta_0.
\]

For \(0<\tau\le\tau_0\), write \(p_j=\Pbb(D=j)\). Since
\[
        1-|\phi_\tau(t)|^2
        =
        \sum_{j,k\ge1}p_jp_k\bigl(1-\cos((j-k)t)\bigr),
\]
the adjacent pair \(j=1,k=2\), together with \(j=2,k=1\), gives
\[
        1-|\phi_\tau(t)|^2
        \ge 2p_1p_2(1-\cos t).
\]
For sufficiently small fixed \(\tau_0\),
\[
        p_1\ge c,
        \qquad
        p_2\ge c\tau,
\]
hence
\[
        1-|\phi_\tau(t)|^2\ge c\tau(1-\cos t)
        \ge c\tau t^2
\]
for \(|t|\le\delta_0\), after reducing \(\delta_0\). This gives
\[
        |\phi_\tau(t)|\le \exp\{-c\tau t^2\}.
\]

For \(\tau_0\le\tau\le T\), compactness gives the same estimate. Indeed,
\[
        \frac{1-|\phi_\tau(t)|^2}{t^2}
\]
extends continuously to \(t=0\), with limiting value \(\sigma_\tau^2>0\).
Because the zero-truncated Poisson law has span one, \(|\phi_\tau(t)|<1\)
for \(0<|t|\le\delta_0\). Hence, after reducing \(\delta_0\) if necessary,
\[
        1-|\phi_\tau(t)|^2\ge c t^2
\]
uniformly for \(\tau_0\le\tau\le T\), and since \(\tau\le T\), this implies
\[
        |\phi_\tau(t)|\le \exp\{-c\tau t^2\}.
\]

It remains to treat \(\tau\ge T\). Put \(q=e^{-\tau}\) and
\[
        \chi_\tau(t):=e^{\tau(e^{it}-1)}.
\]
Then
\[
        \phi_\tau(t)=\frac{\chi_\tau(t)-q}{1-q}.
\]
Let
\[
        a:=\tau(1-\cos t),
        \qquad
        |\chi_\tau(t)|=e^{-a}.
\]
Choose \(\delta_0\) small enough that \(1-\cos t\le 1/4\) for
\(|t|\le\delta_0\). Then
\[
        \frac{q}{e^{-a}}
        =
        e^{-\tau+a}
        \le e^{-3\tau/4}.
\]
If \(a\ge 8e^{-3\tau/4}\), then, for \(T\) sufficiently large,
\[
        |\phi_\tau(t)|
        \le
        \frac{e^{-a}+q}{1-q}
        \le
        e^{-a}(1+2e^{-3\tau/4})
        \le
        e^{-a/2}
        \le
        e^{-c\tau t^2}.
\]
If \(a<8e^{-3\tau/4}\), then
\(|t|=O(e^{-3\tau/8}\tau^{-1/2})\). The Taylor expansion at the origin gives
\[
        |\phi_\tau(t)|^2
        =
        1-\sigma_\tau^2t^2+O(\Ebb D^3\,|t|^3).
\]
Since \(\Ebb D^3=O(\tau^3)\) and \(\sigma_\tau^2\asymp\tau\) for
large \(\tau\), the error term is \(o(\sigma_\tau^2t^2)\), uniformly in
this subcase. Therefore, for \(T\) large enough,
\[
        |\phi_\tau(t)|^2
        \le
        1-\frac12\sigma_\tau^2t^2
        \le
        \exp\{-c\tau t^2\}.
\]
This completes the proof of the small-arc damping estimate.

We also need an off-zero estimate. For every fixed \(\delta\in(0,\pi]\),
there is \(c_\delta>0\) such that
\[
        \sup_{\delta\le |t|\le\pi}|\phi_\tau(t)|^N
        \le
        \exp\{-c_\delta N\min(\tau,1)\}.
\]
For \(0<\tau\le\tau_0\), the adjacent-pair bound above gives
\[
        1-|\phi_\tau(t)|^2\ge c_\delta\tau,
        \qquad \delta\le |t|\le\pi,
\]
and hence
\[
        |\phi_\tau(t)|^N\le \exp\{-c_\delta N\tau\}.
\]
For \(\tau_0\le\tau\le T\), compactness and span one give
\[
        \sup_{\tau_0\le\tau\le T,\ \delta\le |t|\le\pi}|\phi_\tau(t)|
        \le 1-\eta_\delta
\]
for some \(\eta_\delta>0\), hence
\[
        |\phi_\tau(t)|^N\le e^{-c_\delta N}.
\]
For \(\tau\ge T\), using again
\[
        \phi_\tau(t)=\frac{\chi_\tau(t)-e^{-\tau}}{1-e^{-\tau}},
        \qquad
        |\chi_\tau(t)|=e^{-\tau(1-\cos t)},
\]
and the fact that \(1-\cos t\ge c_\delta\) on
\(\delta\le |t|\le\pi\), we obtain
\[
        |\phi_\tau(t)|\le C e^{-c_\delta\tau}.
\]
Increase \(T\) further, if necessary, so that \(\log C\le c_\delta T/2\).
Then, for \(\tau\ge T\),
\[
        |\phi_\tau(t)|^N
        \le
        \exp\{N\log C-c_\delta N\tau\}
        \le
        \exp\{-c_\delta N\tau/2\}.
\]
Renaming \(c_\delta/2\) as \(c_\delta\), this is the required bound.
The off-zero estimate follows in all regimes.

We now evaluate the Fourier integral. Put \(t=s/\sqrt B\). For fixed \(s\),
the centered characteristic function has the expansion
\[
        \psi_\tau(s/\sqrt B)
        =
        1-\frac{\sigma_\tau^2s^2}{2B}
        +
        O\left(\rho_3(\tau)\frac{|s|^3}{B^{3/2}}\right).
\]
Consequently
\[
        \psi_\tau(s/\sqrt B)^N
        =
        \exp\{-s^2/2+o(1)\},
\]
because
\[
        \frac{N\rho_3(\tau)}{B^{3/2}}
        \le
        \frac{C N\sigma_\tau^2(1+\tau)^2}{(N\sigma_\tau^2)^{3/2}}
        =
        \frac{C(1+\tau)^2}{\sqrt B}.
\]
If \(0<\tau\le1\), then \(B\asymp N\tau\to\infty\), so the last expression
is \(O((N\tau)^{-1/2})=o(1)\). If \(1\le\tau\le C_0\log N\), then
\(B\asymp N\tau\), and hence
\[
        \frac{(1+\tau)^2}{\sqrt B}
        \le
        C\frac{\tau^{3/2}}{\sqrt N}
        \le
        C\frac{(\log N)^{3/2}}{\sqrt N}
        \to0.
\]

The small-arc damping estimate gives, for \(|s|\le\delta_0\sqrt B\),
\[
        |\psi_\tau(s/\sqrt B)|^N
        =
        |\phi_\tau(s/\sqrt B)|^N
        \le
        \exp\{-cs^2\}.
\]
Therefore dominated convergence yields
\[
\begin{aligned}
        \int_{-\delta_0}^{\delta_0}\psi_\tau(t)^N\,dt
        &=
        \frac1{\sqrt B}
        \int_{-\delta_0\sqrt B}^{\delta_0\sqrt B}
        \psi_\tau(s/\sqrt B)^N\,ds  \\
        &=
        \frac1{\sqrt B}
        \int_{-\infty}^{\infty}e^{-s^2/2}\,ds\,(1+o(1)) \\
        &=
        \sqrt{\frac{2\pi}{B}}\,(1+o(1)).
\end{aligned}
\]

The complementary arc is negligible. By the off-zero estimate,
\[
        \int_{\delta_0\le |t|\le\pi}|\psi_\tau(t)|^N\,dt
        \le
        2\pi\exp\{-cN\min(\tau,1)\}.
\]
If \(0<\tau\le1\), then \(B\asymp N\tau\to\infty\), and
\[
        \exp\{-cN\tau\}=o(B^{-1/2}).
\]
If \(1\le\tau\le C_0\log N\), then \(B\asymp N\tau\le C N\log N\), so
\[
        \exp\{-cN\}=o(B^{-1/2}).
\]
Hence
\[
        \int_{-\pi}^{\pi}\psi_\tau(t)^N\,dt
        =
        \sqrt{\frac{2\pi}{B}}\,(1+o(1)).
\]
Fourier inversion gives
\[
        \Pbb(D_1+\cdots+D_N=m)
        =
        \frac{1+o(1)}{\sqrt{2\pi B}}.
\]

This proves the local asymptotic when \(B\to\infty\). Since
\[
        B=N\sigma_\tau^2\le C N\tau\le C N\log N\le C m\log m,
\]
we also get, for all sufficiently large \(N\),
\[
        \Pbb(D_1+\cdots+D_N=m)\ge cB^{-1/2}\ge c\,m^{-C}.
\]

Finally consider the remaining bounded-\(B\) regime. Since \(\sigma_\tau^2\asymp\tau\), this
forces \(\tau=O(1/N)\). Therefore
\[
        k:=m-N
        =
        N\left(\frac{\tau}{1-e^{-\tau}}-1\right)
        =
        N\left(\frac{\tau}{2}+O(\tau^2)\right)
        =
        O(1).
\]
The event that exactly \(k\) of the variables equal \(2\), and all remaining
variables equal \(1\), implies
\[
        D_1+\cdots+D_N=N+k=m.
\]
Its probability is
\[
        \binom Nk p_2^k p_1^{N-k},
        \qquad
        p_j:=\Pbb(D=j).
\]
For \(\tau=O(1/N)\),
\[
        p_1=1-\frac{\tau}{2}+O(\tau^2),
        \qquad
        p_2=\frac{\tau}{2}+O(\tau^2).
\]
Since \(k=O(1)\), the displayed probability is bounded below by a negative
power of \(m\). Hence
\[
        \Pbb(D_1+\cdots+D_N=m)\ge c\,m^{-C}
\]
also in the bounded-\(B\) case.
\end{proof}

\end{document}